\documentclass{llncs}
\usepackage{amsmath}
\usepackage{amssymb}
\usepackage{enumerate}
\usepackage{varioref}
\usepackage{charter,eulervm}
\usepackage{graphicx}

\title{{\sc A Note on ``Folding Wheels and Fans''}}

\author{
 Ton~Kloks\inst{1} 
\and
 Yue-Li~Wang\inst{2}} 
\institute{ 
 Department of Computer Science\\
 National Tsing Hua University, Taiwan
\and 
 Department of Information Management\\
 National Taiwan University of Science and Technology, Taiwan\\
 {\tt ylwang@cs.ntust.edu.tw}
}

\begin{document}

\maketitle

\begin{abstract}
In~\cite{kn:gervacio} 
Gervacio, Guerrero and Rara obtained formulas for the size of the 
largest clique onto which a graph $G$ folds. 
We prove an interpolation lemma which simplifies some of their 
arguments and we correct a mistake in their formula for wheels.  
\end{abstract}

\section{Introduction}

We consider undirected graphs without loops or multiple edges. 

\begin{definition}
Let $G=(V,E)$ be a graph and let $x$ and $y$ be two 
vertices in $G$ that are at distance two. 
A simple fold with respect to $x$ and $y$ is the operation 
which identifies $x$ and $y$.  
\end{definition}
If multiple edges appear then they are identified. 

\bigskip 

When $G$ is connected then any maximal sequence of simple folds 
turns $G$ into a clique. 
Cook and Evans, and later also Wood, show that a connected 
graph $G$ can always be folded onto a clique with $\chi(G)$ 
vertices~\cite{kn:cook,kn:wood}. 
We are interested in the number $\Sigma(G)$ 
of vertices of the largest clique 
onto which a connected graph $G$ folds. 

\bigskip 

We follow the terminology of~\cite{kn:hahn}. An 
epimorphism is a surjective homomorphism.  
Consider an epimorphism 
$\phi: G \rightarrow K_s$. 
We call 
\[\{\;\phi^{\leftarrow}(h)\;|\; h \in V(K_s)\;\}\] 
the color classes of $\phi$ in $G$. 
A homomorphism $\phi: G \rightarrow H$ is faithful 
if $\phi(G)$ is an induced subgraph of $H$. A faithful epimorphism 
is called complete.  
 
\bigskip

\begin{definition}
Let $G$ be a graph. The achromatic number $\Psi(G)$ is the 
number of vertices in a largest clique $K$ for which there 
is a complete homomorphism \[f:G \rightarrow K.\] 
\end{definition}

Equivalently, $\Psi(G)$ is the maximal number of colors in a 
proper vertex coloring of $G$ such that for every pair of colors 
there is an edge of which the endvertices are colored with the 
two colors. 
Finding the achromatic number of a graph 
is NP-complete, even for trees, 
however it is fixed-parameter 
tractable~\cite{kn:farber,kn:manlove,kn:mate}. 

\bigskip 

\begin{lemma}
\label{reduction}
Assume that $G$ has a universal vertex $u$. 
Then 
\[\Sigma(G)=1+\Psi(G-u)=\Psi(G).\]
\end{lemma}
\begin{proof}
Any two nonadjacent vertices of $G-u$ are at 
distance two in $G$. Thus any achromatic coloring of 
$G$ is obtained by a maximal series of folds. 
The universal vertex must appear in a color class 
by itself. 

\medskip 

\noindent
Harary and Hedetniemi showed that, if $G$ is the join of two 
graphs $G_1$ and $G_2$, then 
$\Psi(G)=\Psi(G_1)+\Psi(G_2)$~\cite{kn:harary2}.   
\qed\end{proof}

\bigskip 

Bodlaender showed that achromatic number is NP-complete for 
trivially perfect graphs~\cite{kn:bodlaender} (see also~\cite{kn:farber}). 
Since the class of trivially perfect graphs is closed under 
adding a universal vertex, by Lemma~\ref{reduction} 
also the folding problem is 
NP-complete for this class of graphs. 

\section{An interpolation theorem}

Gervacio et al~\cite[Theorem 4.2 and Theorem 5.2]{kn:gervacio} 
showed that, for fans and wheels, 
there is a fold onto $K_k$ for any 
\[\chi(G) \leq k \leq \Sigma(G).\] 
The following theorem is similar to the 
interpolation theorem of Harary, Hedetniemi and Prins~\cite{kn:harary}.

\begin{theorem}
\label{interpolation}
Let $G$ be a connected graph and let $\chi(G) \leq k \leq \Sigma(G)$.
There is a folding of $G$ onto $K_k$.
\end{theorem}
\begin{proof}
The proof is similar to the one given
in~\cite[Proposition~4.12]{kn:hahn} for the achromatic number.

\medskip

\noindent
Consider a simple fold, which maps $G$ to $G^{\prime}$.
Then (see eg~\cite[Corollary~4.6]{kn:hahn}),
\[\chi(G) \leq \chi(G^{\prime}) \leq \chi(G)+1.\]
The folding of $G$ onto the complete graph with $\Sigma(G)$
vertices is a sequence of simple folds. Therefore,
after some initial sequence of simple folds, there must appear
a graph $G^{\prime\prime}$ with chromatic number $k$.
There is a fold of
$G^{\prime\prime}$ onto $K_k$~\cite{kn:cook,kn:wood}. 
The concatenation gives the
desired folding of $G$ onto $K_k$.
\qed\end{proof}

\section{Wheels}

Consider $C_9$ and label the vertices $[1,\dots,9]$. 
We obtain a complete coloring with 
four colors $a$, $b$, $c$ and $d$ by coloring 
the vertices in order 
\[[a,\;d,\;b,\;a,\;c,\;d,\;a,\;c,\;b].\] 
Thus $\Psi(C_9) \geq 4$. However, the formula 
in~\cite[Theorem 6.1]{kn:gervacio} gives $\Psi(C_9)=3$ (for 
$s=1$ and $n=2(s+1)^2+1=9$). 

\bigskip 

The following result of Marcu provides the upperbound. 
It shows that $n \geq 10$ when $\Psi(C_n) = 5$. 
See also~\cite{kn:geller,kn:lee}. 

\begin{theorem}[\cite{kn:marcu}]
For an $n$-cycle with achromatic number $\Psi$, 
\[n \geq 
\begin{cases}
\frac{\Psi(\Psi-1)}{2} & \quad\text{if $\Psi$ is odd}\\
\frac{\Psi^2}{2} & \quad\text{if $\Psi$ is even.}
\end{cases}\] 
\end{theorem}
 
\subsection{Threshold graphs}

A graph is trivially perfect if it has no 
induced $C_4$ and no induced $P_4$. The result of Bodlaender 
implies that the folding number is NP-complete for 
trivially perfect graphs~\cite{kn:bodlaender}. 
In the following theorem we show that if in 
every induced subgraph there is at most one component 
with more than one vertex, then the problem is polynomial. 

\bigskip 

For our purposes the following characterization of 
threshold graphs is most useful. 

\begin{theorem}[\cite{kn:chvatal}]
\label{thm threshold}
A graph is a threshold graph if and only if every 
induced subgraph has an isolated vertex or a universal vertex. 
\end{theorem}

Alternatively, threshold graphs are those graphs without induced 
$P_4$, $C_4$ and $2K_2$. 

\bigskip 

Theorem~\ref{thm threshold} immediately gives the following result. 

\begin{theorem}
If $G$ is a threshold graph then 
\[\chi(G)=\Sigma(G)=\Psi(G).\]
\end{theorem}
\begin{proof}
By a result of Harary and Hedetniemi~\cite[Proposition 6]{kn:harary2}, 
when $G$ is the join of two graphs $G_1$ and $G_2$ 
then 
\begin{equation}
\label{eqnthreshold2}
\Psi(G)=\Psi(G_1)+\Psi(G_2). 
\end{equation}

\medskip 

\noindent
Now assume that $G$ has an isolated 
vertex $x$. We claim that  
\begin{equation}
\label{eqnthreshold}
\Psi(G) = \max\;\{\;1,\;\Psi(G-x)\;\}.
\end{equation}
To see that, observe that $x$ cannot form a color class 
by itself.  
\qed\end{proof}

\end{document}